\documentclass[11pt, reqno]{amsart}
\usepackage{version}
\usepackage{amssymb}
\usepackage[]{enumitem}

\theoremstyle{definition}
\newtheorem{thm}{Theorem}[section]
\newtheorem{cor}{Corollary}[section]
\newtheorem{lem}{Lemma}[section]
\newtheorem{prop}{Proposition}[section]
\newtheorem{df}{Definition}[section]

\newtheorem{rk}{Remark}[section]

\usepackage{hyperref}
\usepackage{lipsum}

\newcommand{\myfootnote}[1]{
    \renewcommand{\thefootnote}{}
    \footnotetext{\hspace{-2pt}\scriptsize#1}
    \renewcommand{\thefootnote}{\arabic{footnote}}
}

\newcommand*{\doi}[1]{\href{http://dx.doi.org/#1}{DOI: #1}}

\usepackage{mathrsfs}
\usepackage{amsmath, amsfonts, amssymb,amsthm}
\newcommand{\field}[1]{\mathbb{#1}}
\newcommand{\N}{\field{N}}

\newcommand{\R}{\field{R}}
\newcommand{\C}{\field{C}}
\newcommand{\B}{\field{B}}

\DeclareMathOperator{\id}{Id}

\DeclareMathOperator{\Aut}{Aut}

\DeclareMathOperator{\supp}{supp}

\begin{document}
\title[]{Families of Proper Holomorphic Embeddings and Carleman-type Theorem with parameters}

\author{Giovanni D. Di Salvo}
\author{Tyson Ritter}
\author{Erlend F. Wold}

%
%
\date{\today}
\keywords{}
\myfootnote
{
    Published in \emph{The Journal of Geometric Analysis},
    January~2023,
    volume~33,
    issue~75.
    }    
\myfootnote
{
	\doi{10.1007/s12220-022-01110-y}.
}
\begin{abstract}
We solve the problem of simultaneously embedding properly holomorphically into $\C^2$ a whole family of $n$--connected domains $\Omega_r\subset\Bbb P^1$ such that none of the components of $\Bbb P^1\setminus\Omega_r$ reduces to a point, by constructing a continuous mapping $\Xi\colon\bigcup_r\{r\}\times\Omega_r\to\C^2$ such that $\Xi(r,\cdot)\colon\Omega_r\hookrightarrow\C^2$ is a proper holomorphic embedding for every $r$. To this aim, a parametric version of both the Anders\'en--Lempert procedure and Carleman's Theorem is formulated and proved.
\end{abstract}

\maketitle

\section{Introduction}

Existence of proper holomorphic embeddings of Riemann surfaces $\mathcal R$ into 2--dimensional complex manifolds $X$, e.g., $X=\mathbb C^2$, 
with prescribed geometrical properties, e.g., being complete, has been an active area of research over the recent years. 
Various techniques have been developed, but in several cases, positive results have been obtained only at the cost 
of perturbing the complex structure of $\mathcal R$ (see \v Cerne--Forstneri\v c \cite{thirdpaper:CF02}, Alarc\'on \cite{thirdpaper:A20} and Alarc\'on--L\'opez \cite{thirdpaper:AL13}). 
It can be hoped, however, that if you let $r$ be a local parameter on the moduli space of Riemann surfaces of a given type, 
and you perform various constructions continuously with the parameter $r$ near a given point $r_0$, then 
you will get a perturbation of the complex structure for each given $r$, but at least one perturbation 
will correspond to your initial $r_0$. Indeed this is the philosophy behind the embedding results of Globevnik--Stens\o nes \cite{thirdpaper:GS95}. 
The purpose of this article is to take a first step towards results of this type that may be generalized to larger classes 
of Riemann surfaces. 

\medskip

We will consider the following. It is known that any $n$--connected domain $\Omega$ in the Riemann sphere may be mapped univalently 
onto a domain in the Riemann sphere whose complement consists of $n$ parallel disjoint 
slits with a given inclination $\Theta$ to the real axis.
The univalent map achieving this is uniquely determined by $\Theta$ and the choice of a certain 
normalization of the Laurent series expansion at a chosen 
point $\zeta\in\Omega$ being sent to $\infty$ (see Goluzin, \cite{thirdpaper:Goluzin}, page 213). Considering a continuous family of $n$--connected domains, we obtain a continuously varying 
family of uniformizing slit--maps. 
\medskip

Let $C_j\subset\C$ be compact disks
and $I_j\subset\R_{>0}$ be compact intervals, $j=1,\dots, n$.
Set $B_j:=C_j\times I_j$ and $B:=B_1\times\cdots\times B_n$.
Let
$r=((a_1,b_1),\dots,(a_n,b_n))$ denote the coordinates on $B$, 
and letting $l_{r,j}$ denote the closed straight line segment 
which is parallel to the real axis with right end--point $a_j(r)$ and of length $b_j(r)$, 
we assume that $L_r:=\{l_{r,1},\dots,l_{r,n}\}$ is a set
of pairwise disjoint slits, and thus $\mathbb P^1\setminus L_r$  is an $n$--connected domain, 
none of whose boundary components are isolated points.
After possibly having to apply the map $z\mapsto(z-a_1(r))/b_1(r)$
we may assume that for all $r$ we have that $l_{r,1}=[-1,0] \subset \C$. \

\medskip

The goal is to prove the following. 
\begin{thm}\label{thirdpaper:thm:main}
In $B\times\mathbb P^1$ set 
$$
\Omega=(B\times\mathbb P^1)\setminus (\bigcup_{r\in B} \{r\}\times L_r).
$$
Then there exists a continuous map $\Xi\colon\Omega\rightarrow\mathbb C^2$ such that 
for each $r\in B$ we have that $\Xi(r,\cdot)\colon\Omega_r\rightarrow\mathbb C^2$ is a 
proper holomorphic embedding. 
\end{thm}

\section{The Setup}\label{thirdpaper:sec:setup}

We will now introduce a setup to prove Theorem \ref{thirdpaper:thm:main}. First, we need the notion of a certain directed family of curves. \

Let $C>0$ and $R>1$. Let $\Gamma$ denote the half line $\Gamma=\{x\in\mathbb R\subset\mathbb C: x\geq R-1\}$,
let $B\subset\R^m$ be a compact set, and denote by $(r,x)$  the coordinates on $B\times\Gamma$.
Let $h, h'=\frac{\partial h}{\partial x}\in \mathscr C(B\times \Gamma)$, and assume 
that 
\begin{equation*}
\left|h(r,x)\right|<\frac C2\;\;,\;\;
\left|h'(r,x)\right|<\frac12.
\end{equation*}

\begin{df}
Let $\theta\in [0,2\pi)$. Then the set of curves 
$$
e^{i\theta}\cdot\{x+ih(r,x)\;:\;r\in B,\; x\in\Gamma\}
$$
is referred to as being \emph{$\theta$--directed}, and subordinate to $R,C$. A family of curves is said to be 
\emph{$\theta$--directed} if it is $\theta$--directed subordinate to $R,C$ for sufficiently large $R,C$. 
\end{df}

With the notation in the previous section, set $\psi(z):=\frac{1}{z}+1$, $\lambda_{r,j}:=\psi(l_{r,j})$, $c_j(r):=\psi(a_j(r))$. Then $\Lambda_r:=\{\lambda_{r,1},\dots,\lambda_{r,n}\}$
is a set of disjoint slits in $\mathbb P^1$, where $\lambda_{r,1}$ is the negative real axis 
and $\lambda_{r,j}$ are circular slits (or possibly straight line segments along the real axis) for $j=2,\dots,n$.  We set $e^{i\theta_{r,j}}:=\psi'(a_j(r))/|\psi'(a_j(r))|$, i.e.,
we have that  $e^{i\theta_{r,j}}$ is a unit tangent to the circle $\Lambda_{r,j}$ on which $\lambda_{r,j}$ lies 
at the point $c_j(r)$. Setting $\alpha_{r,j}(z):= e^{-i\theta_{r,j}}(z - c_j(r))$ we have that 
$\alpha_{r,j}(\Lambda_{r,j})$ is a circle which is tangent to the  
real axis at the origin, 
and we let $\kappa_{r,j}$ denote the signed curvature of this circle; positive if the 
circle is in the upper half plane, negative if the circle is in the lower half plane, and zero if the circle is the real axis. 

\medskip

\begin{prop}\label{thirdpaper:tang1}	
Fix $j \in \{2,\dots,n\}$ and suppose that $g_{r,j}\in\mathcal O(\triangle_\delta(c_j(r)))$ is a continuous family of functions, for $r\in B$. Let $\theta\in [0,2\pi)$, and set 
$$
\varphi_j(r,z) := \frac{e^{i\theta}}{\alpha_{r,j}(z)} + g_{r,j}(z).
$$
Then the family $\Gamma_j$ of curves  $\varphi(r,\lambda_{r,j})$ is $(\theta-\pi)$--directed. 
\end{prop}
\begin{proof}
It suffices to prove this for $\theta=0$.  Then $\alpha_{r,j}(\Lambda_{r,j})$ is parametrized near the origin by 
$$
\eta_{r,j}(x)= x + i\frac{\kappa_{r,j}}{2}x^2 + O(x^4).
$$
Set $\tilde g_{r,j}(z)= g_{r,j}(\alpha_{r,j}^{-1}(z))$ 
We have that 
\begin{align*}
\varphi_j(r,x)& =\frac{1}{x + i\frac{\kappa_{r,j}}{2}x^2 + O(x^4)} + \tilde g_{r,j}(\eta_{r,j}(x)) \\
& = \frac{x - i\frac{\kappa_{r,j}}{2}x^2 + O(x^4)}{x^2 +  O(x^4)} + \tilde g_{r,j}(\eta_{r,j}(x)) \\
& = \left(\frac{1}{x} - i\frac{\kappa_{r,j}}{2} + O(x^2)\right)(1+O(x^2)) + \tilde g_{r,j}(\eta_{r,j}(x))\\
& = \frac{1}{x} - i\frac{\kappa_{r,j}}{2} + O(x) + \tilde g_{r,j}(\eta_{r,j}(x)).
\end{align*}
Since $g_{r,j}(z)$ is close to a constant when $z$ is close to $c_j(r)$, the uniform bound in the 
definition of $(-\pi)$--directed holds.  Now 
$$
\varphi_j'(r,x) = \frac{-1}{x^2} + v_{r,j}(x),
$$
where $v_{r,j}(x)$ is bounded and scaling it to have almost unit length we see 
$$
x^2\varphi_j'(r,x) = -1 + x^2 v_{r,j}(x).
$$
\end{proof}

\begin{prop}\label{thirdpaper:t-tang}
Fix $\theta_2,\dots,\theta_n\in (0,2\pi)$. 
Define $\phi_r\colon\mathbb C\setminus\{c_2(r),\dots,c_n(r)\}\rightarrow \mathbb C^2$ by 
$$
\phi_r(z) := \left(z, \sum_{j=2}^n \frac{e^{i\theta_j}}{\alpha_{r,j}(z)}\right).
$$
Choose $\delta>0$ small, and let $a,b\in\triangle_\delta(1/\sqrt 2)$, 
and set $A_{a,b}(z,w):=(az+bw,-bz+aw)$. Write $a=r_ae^{i\vartheta_a},  b=r_be^{i\vartheta_b}$.
Then the family $\Gamma_1$ defined by $\Gamma_1=\{\pi_1\circ A_{a,b}\circ\phi_r(\lambda_{r,1}):r\in B\}$ is $(\vartheta_a-\pi)$--directed, 
and each family $\Gamma_j, j=2,\dots,n$, defined by 
$\Gamma_j=\{\pi_1\circ A_{a,b}\circ\phi_r(\lambda_{r,j}):r\in B\}$ is $(\vartheta_b+\theta_j-\pi)$--directed. 
\end{prop}
\begin{proof}
For $j=2,\dots,n$ this is just Proposition \ref{thirdpaper:tang1} since for any fixed $j$ we have that $\pi_1\circ A_{a,b}\circ\phi_r(\lambda_{r,j})$
is parametrized by 
$$
\frac{r_be^{i{(\vartheta_b+\theta_j)}}}{\alpha_{r,j}(z)} + \sum_{k\neq j}\left(\frac{r_be^{i(\vartheta_b+\theta_k)}}{\alpha_{r,k}(z)}\right) + r_ae^{i\vartheta_a}z.
$$
For $j=1$ this is because $\pi_1\circ A_{a,b}\circ\phi_r(\lambda_{r,j})$ is parametrized by  
$r_ae^{i\vartheta_a}z + g_r(z)$ where $g_r(z)$ is uniformly comparable 
to $\frac{1}{z}$.
\end{proof}

\section{Carleman approximation with parameters}

We will start by introducing some notation. Afterwards, we present Theorem \ref{thirdpaper:Carleman}, a Carleman--type theorem (see e.g., \cite{thirdpaper:FornaessForstnericWold}), which is the main result of the present section: families of smooth functions holomorphic on a disc can be approximated by entire functions on a smaller disc and on the union of several Lipschitz curves. The proof is obtained applying inductively Corollary \ref{thirdpaper:merg}, which in turn easily follows from Proposition \ref{thirdpaper:cousin}, a tool that allows to approximate smooth functions on compact pieces of a Lipschitz curve; Corollary \ref{thirdpaper:merg} extends the result to several curves. Proposition \ref{thirdpaper:cousin} relies on three technical lemmata that will be presented in Section \ref{thirdpaper:sec:mergelyanrunge}.


\subsection{The setup}
Recall that $R>1$, $\Gamma$ is the half line $\Gamma:=\{x\in\Bbb R\subset\Bbb C\;:\;x\ge R-1\}$, $B\subset\R^m$ is a compact and $(r,x)$ are the coordinates on $B\times\Gamma$.
For $k=1,\dots,n$ let $h_k,h'_k=\frac{\partial h_k}{\partial x}\in\mathscr C(B\times \Gamma)$ be such that 
\begin{equation}\label{thirdpaper:lip1/4}
\left|h_k(r,x)\right|<\frac C2\;\;,\;\;
\left|h_k'(r,x)\right|<\frac12
\end{equation}
for some $C>0$, for every $(r,x)\in B\times\Gamma$, and every $k=1,\dots,n$. Then, setting $l=1/2$, we have that
\begin{align}\label{thirdpaper:lip}
|h_k(r,x_1)-h_k(r,x_2)|\le l|x_1-x_2|,\;\;\forall x_1,x_2\in\Gamma,\;r\in B\;,
\end{align}
so $h_k$ is $l$--Lipschitz and in this way we also call its graph.
Let $0=\theta_1<\theta_2<\cdots <\theta_n<2\pi$ and define the Lipschitz curves
\begin{align*}
\Gamma_{k,r}&:=e^{i\theta_k}\cdot\{x+ih_k(r,x)\;:\;x\in\Gamma\}\;
\end{align*}
and their union
\begin{align*}
\Gamma_r&:=\bigcup_{k=1}^n\Gamma_{k,r}\;.
\end{align*}
If $D\subseteq\Omega\subseteq\C$ are domains, a useful notation is given by setting
$$
\mathcal P(B,\Omega,D):=\{f\in\mathscr C(B\times\Omega)\;:\;f(r,\cdot)\in\mathcal O(D)\;\;\forall r\in B\}
$$
and
$$
\mathcal P(B,\Omega):=\mathcal P(B,\Omega,\Omega)\;.
$$

\begin{thm}[Carleman--type Theorem with parameters]\label{thirdpaper:Carleman}
Assume that $f \in\mathcal P(B,\Bbb C,\overline\triangle_{\rho+3+\frac{3C}{2}})$ for some $\rho>R$. Then for any $\epsilon\in  \mathscr C(\mathbb C), \epsilon>0$, 
there exists $g\in \mathcal P(B, \mathbb C)$ such 
that
$$
|g(r,z) - f(r,z)|<\epsilon (z)
$$
for all $z\in\overline\triangle_{\rho}\cup \Gamma_r,\;\; r\in B$. 

\end{thm}

\subsection{Proof of Theorem \ref{thirdpaper:Carleman}}

Fix $j\in\N$, $j\geq R$ and let $b$ be some real number such that 
$$
j+3+\frac{3C}{2}<b\;.
$$ 
For $\rho\ge C$ set 
\begin{align*}
\psi(\rho)&:=\arcsin\frac{C}{\rho}
\end{align*}
and define
\begin{align*}
S_{\rho}:=\{se^{i\theta}: 0<s<\infty, |\theta|<\psi(\rho)\}\;\; \mbox{ and }\;\; A_{\rho,b} := \triangle_b\setminus\overline{S_\rho}.
\end{align*}
Then $S_\rho$ is the wedge in the right half--plane bounded by the straight lines passing through 
the origin and the intersection between $\partial\triangle_\rho$ and the lines $y=\pm C$. 
Up to consider a larger $R$, we assume $e^{i\theta_j}S_{\rho}\cap e^{i\theta_k}S_{\rho}=\emptyset$ for all $j\neq k$ for $\rho\geq R$. 
We define further the following sets
\begin{align*}
\omega_1&:=\{z=x+iy: j+1<x, |z|<b, |y|<C\}\\
\omega_2&:= \{z=x+iy: 0<x<j+2, |y|<C\}\cup A_{j,b} \\
\Omega&:=\omega_1\cup\omega_2
\end{align*}
Given $\delta>0$, we will denote the open $\delta$--neighborhood of $D$ as
$$
D(\delta):=\{z\in\C\;:\;d(z,D)<\delta\}\;.
$$
The following proposition, or rather its corollary below, is the main technical ingredient in the proof of 
the Carleman Theorem \ref{thirdpaper:Carleman}. The proposition follows from Lemma \ref{thirdpaper:manne}, Lemma \ref{thirdpaper:appr2}, and 
finally Lemma \ref{thirdpaper:runge} below. 

\begin{prop}\label{thirdpaper:cousin}
Assume that $n=1$. 
Let $\alpha\colon\bigcup_{r\in B}\{r\}\times \Gamma_r\to\C$ be continuous such that $\alpha(r,\cdot)\in\mathscr C_c(\Gamma_r)$ for every $r\in B$, with
$$
\supp \alpha(r,\cdot)\subset\{z=x+iy\in\Gamma_r\;:\; j+3+\frac{3C}{2} < x, |z|<b,\;|y|<C/2\}\;\;\;\forall r\in B. \;
$$
Then for every $\epsilon>0$ there exists $\{Q_t\}_{t>0}\subset\mathcal P(B,\mathbb C)$ such that
\begin{equation}\label{thirdpaper:cousapprox}
\|\alpha(r,\cdot)-Q_t(r,\cdot)\|_{\Gamma_r\cap \overline\triangle_{b}}<\epsilon
\end{equation}
for every $r\in B$, $0<t<t_0$, and 
\begin{equation}\label{thirdpaper:cousvanish}
Q_t\to0\;\;\mbox{as}\;\;t\to0
\end{equation}
uniformly on $B\times\omega_2(\delta)$, for some $\delta>0$.
\end{prop}

\vskip 2cm

\begin{cor}\label{thirdpaper:merg}
Let $\alpha\colon\bigcup_{r\in B}\{r\}\times \Gamma_{r}\to\C$ be continuous such that $\alpha(r,\cdot)\in\mathscr C_c(\Gamma_r)$ for every $r\in B$, with
$$
\supp \alpha(r,\cdot)\subset\{z\in\Gamma_r\;:\; j+3+\frac{3C}{2} < |z|<b\}\;,\;\;\forall r\in B\;.
$$
Then for every $\epsilon>0$ there exists $\{Q_t\}_{t>0}\subset\mathcal P(B,\mathbb C)$ such that
\begin{equation}\label{thirdpaper:cousapprox}
\|\alpha(r,\cdot)-Q_t(r,\cdot)\|_{\Gamma_r\cap \overline\triangle_{b}}<\epsilon
\end{equation}
for every $r\in B$, $0<t<t_0$, and 
\begin{equation}\label{thirdpaper:cousvanish}
Q_t\to0\;\;\mbox{as}\;\;t\to0
\end{equation}
uniformly on $B\times\overline\triangle_j(\delta)$, for some $\delta>0$. 
\end{cor}

\begin{proof}
On $e^{-i\theta_k}\Gamma_{k,r}$ define $\alpha_k(r,z):=\alpha(r,e^{i\theta_k}z)$. 
Using the proposition we obtain approximations $Q_{t,k}(r,z)$. Then setting
$$
Q_t(r,z) := \sum_{k=1}^n Q_{t,k}(r,e^{-i\theta_k}z)
$$
will yield the result for sufficiently small $t$.
\end{proof}

\emph{Proof of Theorem \ref{thirdpaper:Carleman}:} The proof is by induction on $k\geq 0$, and the induction hypothesis is the following. 
For every $j=0,\dots,k$ there exist:
\begin{enumerate}[label=(\roman*)]
\item  $g_j\in\mathcal P(B,\Bbb C,\overline{\triangle}_{\rho +j+3+\frac{3C}{2}})$ \label{thirdpaper:gstaysinP},
\item $|g_j(r,z)-f(r,z)|<\epsilon(z)/2$ for all $z\in\overline\triangle_{\rho}\cup\Gamma_r$, $r\in B$,\;and \label{thirdpaper:gclosetof}
\item $\|g_j-g_{j-1}\|_{B\times\overline{\triangle}_{\rho+j-1}}<2^{-j}$ for $j\geq 1$. \label{thirdpaper:gconverges}
\end{enumerate}
We start by setting $g_0:=f$; then in the case $k=0$ we see that \ref{thirdpaper:gstaysinP}, \ref{thirdpaper:gclosetof} hold, and \ref{thirdpaper:gconverges} is void.  \
Assume now that the induction hypothesis holds for some $k\geq 0$.  Fix 
$\eta>0$ such that 
$$
g_k(r,\cdot)\in\mathcal O(\overline{\triangle}_{\eta+\rho+k+3+\frac{3C}{2}}),
$$
and choose a cutoff function $\chi\in \mathscr C^\infty(\mathbb C)$ such that 
$0\leq \chi\leq 1$, such that $\chi=0$ near $\overline{\triangle}_{\rho +k+3+\frac{3C}{2}}$, 
and $\chi=1$ outside $\overline{\triangle}_{\eta+\rho+k+3+\frac{3C}{2}}$.
Now $g_k$ may be approximated on $\overline{\triangle}_{\eta+\rho+k+3+\frac{3C}{2}}$
to arbitrary precision by $h_k\in\mathscr C(B)[z]$ using 
Taylor series expansion, and so 
$$
h_k + \chi\cdot (g_k - h_k) = :  h_k + \alpha_k
$$
approximates $g_k$ to arbitrary precision. Hence it suffices
to approximate $\alpha_k$ to arbitrary precision by a suitable function.  
Multiplying $\alpha_k$ by a suitable cutoff function so that Corollary \ref{thirdpaper:merg} applies, we have that $\alpha_k$ may be approximated to arbitrary precision on 
$$
\bigcup_r\Gamma_r\cap \overline{\triangle}_{\rho+k+2+3+\frac{3C}{2}}
$$
by a function $Q_k\in\mathcal P(B,\mathbb C)$ which is arbitrarily small on $\overline{\triangle}_{\rho+k}$. Setting then $g_{k+1} := h_k + Q_k + \tilde\chi\cdot (\alpha_k - Q_k)$ where $\tilde\chi$ is a third cutoff function such that $\tilde\chi=0$ near $\overline{\triangle}_{\rho+k+1+3+\frac{3C}{2}}$ and $\tilde\chi=1$ near $\Bbb C\setminus \triangle_{\rho+k+2+3+\frac{3C}{2}}$, completes the induction step. We may finish the proof of Theorem \ref{thirdpaper:Carleman}
by setting $g:=\lim_{j\rightarrow\infty} g_j$, which exists by \ref{thirdpaper:gconverges}, and the approximation holds by \ref{thirdpaper:gclosetof}. 
$\hfill\blacksquare$

\subsection{Lemmata: Mergelyan--type and Runge's Theorems with parameters}\label{thirdpaper:sec:mergelyanrunge}

The three lemmata we present and prove in this section are fundamental ingredients to formulate a Mergelyan--type Theorem (see e.g., \cite{thirdpaper:FornaessForstnericWold}). The first one of them generalizes a theorem proved by P. Manne in his Ph.D. thesis \cite{thirdpaper:Ma93} and is about the holomorphic (entire) approximation of a family of smooth functions, each of which is defined on a Lipschitz curve in the complex plane.
\newline
\newline

\begin{lem}\label{thirdpaper:manne}
Assume that $n=1$, and let $\alpha$ be as in Proposition \ref{thirdpaper:cousin}.
Then for every $\epsilon>0$ there exists $\{H_t\}_{t>0}\subset\mathcal P(B,\mathbb C)$ such that
\begin{equation}\label{thirdpaper:manneapprox}
\|\alpha(r,\cdot)-H_t(r,\cdot)\|_{\Gamma_r}<\epsilon
\end{equation}
for every $r\in B$, $0<t<t_0$, for some $t_0>0$, and 
\begin{equation}\label{thirdpaper:small}
H_t\to0\;\;\mbox{as}\;\;t\to0
\end{equation}
uniformly on $B\times (\omega_1\cap \omega_2)(\delta)$, for some $\delta>0$. 
\end{lem}
\begin{proof}
Extend $h_1$ to a function $h$ on the whole real line by setting $h(r,x):=h_1(r,x)$ for $x\geq R-1$ and 
$h(r,x):=h_1(r,2R-2-x)$ for $x<R-1$. Define $S_{r}:=\{s+ih(r,s)\;:\;s\in\Bbb R\}$. Denote by $z=x+ih(r,x)$ a point in $\Gamma_r$ and let $\zeta=\zeta(r,s)=s+ih(r,s)$ be a parametrization of $S_r$. Further, $\zeta'(r,s)=\frac{\partial\zeta}{\partial s}(r,s)$, extend $\alpha(r,\cdot)$ to  $S_r\setminus\Gamma_r$ to be 0 for all $r\in B$ and define
\begin{align*}
H_t(r,z)
&:=\int_{S_r}\alpha(r,\zeta)K_t(\zeta,z)\,d\zeta\\
&=\int_{\Bbb R}\alpha(r,\zeta(r,s))K_t(\zeta(r,s),z)\zeta'(r,s)\,ds
\end{align*}
for $t>0,r\in B,z\in\C$, where
$$
K_t(\zeta,z):=\frac1{t\sqrt\pi}e^{-\frac{(\zeta-z)^2}{t^2}}\;
$$
is the Gaussian kernel.  \

We start by proving (\ref{thirdpaper:small}). Let $z=x+iy\in(\omega_1\cap \omega_2)(\delta)$. We have that 
\begin{align*}
|H_t(r,z)|
&\le\frac1{t\sqrt\pi}\int_{\Bbb R}|\alpha(r,\zeta(r,s))|e^{-\frac{(s-x)^2-(h(r,s)-y)^2}{t^2}}\left|\zeta'(r,s)\right|\,ds\\
&=\frac1{t\sqrt\pi}\int_{\substack{j+3+\frac32C<s<b}}|\alpha(r,\zeta(r,s))|e^{-\frac{(s-x)^2-(h(r,s)-y)^2}{t^2}}\left|\zeta'(r,s)\right|\,ds\;,\\
\end{align*}
and $(s-x)^2 - (h(r,s)-y)^2\geq (1+\frac{3C}{2}-\delta)^2 - (\frac{3C}{2})^2$, therefore (\ref{thirdpaper:small}) follows.  \

\medskip

For any fixed $\eta>0$ we split $S_r$ as
\begin{align*}
S_r^{(1)}&:=\{\zeta\in S_r\;:\;|\Re(\zeta-z)|\le\eta\}=\{\zeta(r,s)\;:\;|s-x|\le \eta\}\\
S_r^{(2)}&:=\{\zeta\in S_r\;:\;|\Re(\zeta-z)|>\eta\}=\{\zeta(r,s)\;:\;|s-x|>\eta \}\;.
\end{align*}
Since by (\ref{thirdpaper:lip}) we have 
$$
|K_t(\zeta,z)|
\le\frac1{t\sqrt\pi}e^{-\frac{(s-x)^2(1-l^2)}{t^2}}\;,
$$
we immediately get the following upper bound:
\begin{align}
\int_{S_r^{(1)}}|K_t(\zeta,z)|\,d|\zeta|
\le&\frac1{t\sqrt\pi}\int_{x-\eta}^{x+\eta}e^{-\frac{(s-x)^2(1-l^2)}{t^2}}\,ds\nonumber\\
=&\frac1{\sqrt{\pi(1-l^2)}}\int_{|u|\le\frac{\sqrt{1-l^2}}{t}\eta}e^{-u^2}\,du\nonumber\\
\le&\frac1{\sqrt{1-l^2}}\;,\label{thirdpaper:gamma1}
\end{align}
which holds for every $z\in\Gamma_r$, $r\in B$ and $t>0$. Similarly, one sees that for all $\epsilon>0,\eta>0$ there exists $t_0>0$ such that
\begin{align}
\int_{S_r^{(2)}}|K_t(\zeta,z)|\,d|\zeta|
\le\frac1{\sqrt{\pi(1-l^2)}}\int_{|u|>\frac{\sqrt{1-l^2}}{t}\eta}e^{-u^2}\,du
<\epsilon\label{thirdpaper:gamma2}
\end{align}
for every $z\in\Gamma_r,\;r\in B,\;0<t<t_0$.  \ 
We need one last property of the kernel, that is
\begin{align}\label{thirdpaper:gaussintegral}
\int_{S_r}K_t(\zeta,z)\,d\zeta=1
\end{align}
for all $z\in\Gamma_r$, $r\in B$ and $t>0$.
Let us consider the function
$$
F(z):=\int_{S_r}K_t(\zeta,z)\,d\zeta=\frac1{t\sqrt{\pi}}\int_{S_r}e^{-\frac{(\zeta-z)^2}{t^2}}\,d\zeta
$$
which is holomorphic entire. Let $z=x\in\Bbb R$ and define for $T>0$
\begin{align*}
A(T):=&\{u+i0\;:\;-T\le u\le T\}\;,\\
S_r(T):=&\{\zeta\in S_r\;:\;-T\le s\le T\}\;,
\end{align*}
and let $\rho_r^{\pm}(T)$ be the straight line segment between $\pm T$ and $\pm T+ih(r,\pm T)$.  Set 
$$
\gamma_r(T):=A(T)+\rho^+_r(T)-S_r(T)-\rho^{-}_r(T)
$$
which is a piecewise $\mathscr C^1$--smooth closed curve which is nullhomotopic, hence we get
$$
\frac1{t\sqrt{\pi}}\int_{\gamma_r(T)}e^{-(\frac{\zeta-x}{t})^2}\,d\zeta
=0\;
$$
for every $t>0,\;r\in B,\;x\in\Bbb R$ and $T>0$. On the other hand
\begin{align*}
\frac1{t\sqrt{\pi}}\int_{\gamma_r(T)}e^{-(\frac{\zeta-x}{t})^2}\,d\zeta
=\frac1{t\sqrt{\pi}}
&\left(\int_{A(T)}e^{-(\frac{u-x}{t})^2}\,du
+\int_{\rho^+_r(T)}e^{-(\frac{\zeta-x}{t})^2}\,d\zeta\right.\\
&\left.-\int_{S_r(T)}e^{-(\frac{\zeta-x}{t})^2}\,d\zeta
-\int_{\rho^-_r(T)}e^{-(\frac{\zeta-x}{t})^2}\,d\zeta
\right)\;.
\end{align*}
Passing to the limit as $T\to+\infty$, the vertical contributions vanish (as $h$ is bounded), while
\begin{align*}
\frac1{t\sqrt{\pi}}\int_{A(T)}e^{-(\frac{u-x}{t})^2}\,du\longrightarrow
\frac1{t\sqrt\pi}\int_{\Bbb R}e^{-(\frac{u-x}{t})^2}\,du=1
\end{align*}
and
\begin{align*}
\frac1{t\sqrt{\pi}}\int_{S_r(T)}e^{-(\frac{\zeta-x}{t})^2}\,d\zeta
\longrightarrow
\frac1{t\sqrt\pi}\int_{S_r}e^{-(\frac{\zeta-x}{t})^2}\,d\zeta=F(x)
\end{align*}
for every $t>0$, $r\in B$ and $x\in\Bbb R$. This implies that the entire function $F$ is identically $1$ on the real line for every $t>0$ and $r\in B$, so by the identity principle it is constantly $1$ on the whole $\Bbb C$; in particular (\ref{thirdpaper:gaussintegral}) holds true.

We gathered all the ingredients to prove (\ref{thirdpaper:manneapprox}). Let $\epsilon>0$, let $\eta>0$ such that $|\alpha(r,\zeta)-\alpha(r,z)|<\epsilon$ for all $z\in\Gamma_r$, $\zeta\in S_r^{(1)}$, for all $r\in B$.  Then
\begin{align*}
\left|H_t(r,z)-\alpha(r,z)\right|
=&\left|\int_{S_r^{(1)}}\alpha(r,\zeta)K_t(\zeta,z)\,d\zeta+\int_{S_r^{(2)}}\alpha(r,\zeta)K_t(\zeta,z)\,d\zeta-\alpha(r,z)\right|\nonumber\\
=&\left|\int_{S_r^{(1)}}\alpha(r,\zeta)K_t(\zeta,z)\,d\zeta+\right.\\
&\hspace{1cm}\left.\int_{S_r^{(2)}}\alpha(r,\zeta)K_t(\zeta,z)\,d\zeta-\alpha(r,z)\int_{S_r}K_t(\zeta,z)\,d\zeta\right|\nonumber\\
\le&\left|\int_{S_r^{(1)}}(\alpha(r,\zeta)-\alpha(r,z))K_t(\zeta,z)\,d\zeta\right|+\\
&\hspace{1cm}\left|\int_{S_r^{(2)}}(\alpha(r,\zeta)-\alpha(r,z))K_t(\zeta,z)\,d\zeta\right|\\
\le&\epsilon\int_{S_r^{(1)}}|K_t(\zeta,z)|\,d|\zeta|+\\
&\hspace{1cm}\left(\|\alpha(r,\cdot)\|_{S_r^{(2)}}+|\alpha(r,z)|\right)\int_{S_r^{(2)}}|K_t(\zeta,z)|\,d|\zeta|\\
\le&\frac{\epsilon}{\sqrt{1-l^2}}+2\epsilon\|\alpha(r,\cdot)\|_{S_r}\nonumber\;,
\end{align*}
where the second equality follows from (\ref{thirdpaper:gaussintegral}) and the last inequality follows from (\ref{thirdpaper:gamma1}) and (\ref{thirdpaper:gamma2}).
So we can conclude, since this last quantity can be taken arbitrarily small for $\epsilon$ small, uniformly in $z\in\Gamma_r$ and $r\in B$, for all $0<t<t_0$, where $t_0$ comes from (\ref{thirdpaper:gamma2}).
\end{proof}

The following Lemma shows how to modify the approximation constructed in Lemma \ref{thirdpaper:manne}, so that, besides approximating the given smooth function, it becomes arbitrarily small on a suitable region. The price to pay is that the approximation obtained this way is no more entire; we will get "entireness" back with Lemma \ref{thirdpaper:runge}, that is a parametric version of Runge's Theorem (see e.g., \cite{thirdpaper:FornaessForstnericWold}).

\begin{lem}\label{thirdpaper:appr2}
Assume $n=1$, and let $\alpha$ be as in Proposition \ref{thirdpaper:cousin}.
Then for every $\epsilon>0$ there exists $\{\xi_t\}_{t>0}\subset\mathcal P(B,\Omega(\delta))$ such that
\begin{equation}\label{thirdpaper:cousapprox}
\|\alpha(r,\cdot)-\xi_t(r,\cdot)\|_{\Gamma_r\cap \overline\triangle_{b}}<\epsilon
\end{equation}
for every $r\in B$, $0<t<t_0$, and 
\begin{equation}\label{thirdpaper:cousvanish}
\xi_t\to0\;\;\mbox{as}\;\;t\to0
\end{equation}
uniformly on $B\times \omega_2(\delta)$, for some $\delta>0$.
\end{lem}

\begin{proof}
Let $H_t$ be the map defined in Lemma \ref{thirdpaper:manne} and $\phi_i\colon \Omega(\delta)\to[0,1]$ be smooth, such that
\begin{enumerate}[label=(\roman*)]
\item \label{thirdpaper:supp} $\supp\phi_i\subset\omega_i(\delta)$, and
\item $\phi_1+\phi_2\equiv1$ on $\Omega(\delta)$
\end{enumerate}
for some $\delta>0$. Define
$$
g_{t,1}(r,z):=-H_t(r,z)\phi_2(z),\;\;g_{t,2}(r,z):=H_t(r,z)\phi_1(z)\;
$$
on $B\times\Omega(\delta)$; then $g_{t,i}(r,\cdot)$ is a smooth function on $\Omega(\delta)$ and
\begin{align}
g_{t,2}-g_{t,1}=H_t\;
\end{align}
holds true on $B\times\Omega(\delta)$, therefore $\frac{\partial g_{t,1}}{\partial\overline z}(r,z)$ and $\frac{\partial g_{t,2}}{\partial\overline z}(r,z)$ are the same function; call it $v_t$ and consider $v_t(r,\cdot)$ smoothly extended on $\overline{\Omega(\delta)}$. Hence, defining
$$
u_t(r,z):=\frac1{2\pi i}\iint_{\Omega(\delta)}\frac{v_{t}(r,\zeta)}{\zeta-z}\;d\zeta\wedge d\overline\zeta\;,
$$
we can assume without loss of generality to have smoothed the corners of $\omega_i$ so that $\Omega(\delta)$ is smoothly bounded, hence we are allowed to apply Theorem 2.2 in \cite{thirdpaper:B15}, which ensures that $u_t(r,\cdot)$ is smooth on $\overline{\Omega(\delta)}$ for every $r\in B$ and solves $\frac{\partial u_t}{\partial\overline z}=v_t$ on $B\times\overline{\Omega(\delta)}$, hence
\begin{equation}\label{thirdpaper:theta}
\Theta_{t,i}:=g_{t,i}-u_{t}\in\mathcal P(B,\Omega(\delta))\;,\;\;i=1,2\;.
\end{equation}
Then (\ref{thirdpaper:small}) and \ref{thirdpaper:supp} imply
\begin{itemize}
\item $g_{t,i}\to0$ uniformly on $P\times \omega_i(\delta)$ as $t\to0$\;,
\item $\frac{\partial g_{t,1}}{\partial\overline z}(r,z)=-H_t(r,z)\frac{\phi_2}{\partial\overline z}(z)\to0$ uniformly on $B\times \omega_1(\delta)$ as $t\to0$\;, and
\item $\frac{\partial g_{t,2}}{\partial\overline z}(r,z)=H_t(r,z)\frac{\phi_1}{\partial\overline z}(z)\to0$ uniformly on $B\times \omega_2(\delta)$ as $t\to0$\;.
\end{itemize}
The last two imply $u_t\to0$ uniformly on $B\times\Omega(\delta)$, hence 
\begin{equation}\label{thirdpaper:theta2}
\Theta_{t,i}\to0
\end{equation}
uniformly on $P\times\omega_i(\delta)$ as $t\to0$.
Since $\Theta_{t,2}-\Theta_{t,1}=H_t$ on $B\times\Omega(\delta)$, it follows from (\ref{thirdpaper:manneapprox}), (\ref{thirdpaper:theta}) and (\ref{thirdpaper:theta2}) that 
\begin{align*}
\xi_t:=
\left\{
\begin{array}{cc}
H_t+\Theta_{t,1}\;\;&B\times\omega_1(\delta)\\
\Theta_{t,2}\;\;&B\times\omega_2(\delta)
\end{array}
\right.
\end{align*}
satisfies the stated properties.\end{proof}

\begin{lem}[Runge--type Theorem with parameters]\label{thirdpaper:runge}
With the notation of the previous lemma, for all $\epsilon>0$, there is $Q_t\in\mathscr C(B)[z]$ (polynomial with coefficients in $\mathscr C(B)$; in particular $\{Q_t\}_{t\ge0}\subset\mathcal P(B,\Bbb C)$) such that
$$
\|\xi_t-Q_t\|_{B\times \overline\Omega}<\epsilon
$$
for all $t>0, r\in B$.
\end{lem}
\begin{proof}
Observe that $\overline\Omega$ is compact polynomially convex. Let $\gamma=\partial\Omega(\frac{\delta}{2})$. 
We have that
$$
(r,\zeta,z)\mapsto\frac{\xi_t(r,\zeta)}{\zeta-z}
$$
is uniformly continuous on $B\times \gamma\times \overline\Omega$, hence for every $\epsilon>0$ there exists $\eta>0$, such that, dividing $\gamma$ into $N$ pieces $\gamma_1,\dots,\gamma_N$ whose length $L(\gamma_j)$ is less than $\eta$ and fixing a point $\zeta_j\in\gamma_j$ for every $j$, 
\begin{align*}
\left|\frac{\xi_t(r,\zeta)}{\zeta-z}-\frac{\xi_t(r,\zeta_j)}{\zeta_j-z}\right|<\frac1N\frac{2\pi}{L(\gamma_j)}\epsilon
\end{align*}
holds $\forall (r,\zeta,z)\in B\times \gamma_j\times \overline\Omega$. Calling $\gamma_j(1),\;\gamma_j(0)$ the final and initial points of $\gamma_j$, for every $(r,z)\in B\times \overline\Omega$ one has
\begin{align*}
\xi_t(r,z)-\overbrace{\sum_{j=1}^N\frac{\gamma_j(1)-\gamma_j(0)}{2\pi i}\frac{\xi_t(r,\zeta_j)}{\zeta_j-z}}^{=:\beta_t(r,z)}
&=\frac1{2\pi i}\int_{\gamma}\frac{\xi_t(r,\zeta)}{\zeta-z}\,d\zeta-\sum_{j=1}^N\frac{1}{2\pi i}\int_{\gamma_j}\frac{\xi_t(r,\zeta_j)}{\zeta_j-z}\,d\zeta\\
&=\frac1{2\pi i}\sum_{j=1}^N\int_{\gamma_j}\left(\frac{\xi_t(r,\zeta)}{\zeta-z}-\frac{\xi_t(r,\zeta_j)}{\zeta_j-z}\right)\,d\zeta
\end{align*}
thus
\begin{align}\label{thirdpaper:one}
\|\xi_t-\beta_t\|_{B\times \overline\Omega}<\epsilon
\end{align}
for every $t>0$.  
The result now follows since each rational function $z\mapsto\frac{1}{\zeta_j-z}$ may be approximated arbitrarily 
well on $\overline{\Omega}$ by polynomials.  
\end{proof}

\section{Anders\'{e}n--Lempert Theory}\label{thirdpaper:sec:ALtheory}

We will now apply Anders\'{e}n--Lempert Theory in $B\times\mathbb C^2$. We have that $B\subset\mathbb R^N$, 
and when talking about analytic properties of sets and functions on $B\times\mathbb C^2$
we will think of $B\times\mathbb C^2\subset\mathbb C^N\times\mathbb C^2$, $\mathbb C^N=\mathbb R^N+i\mathbb R^N$. 
For instance, by saying that $K\subset B\times\mathbb C^2$ is polynomially convex compact we mean polynomially 
convex in $\mathbb C^N\times\mathbb C^2$; this is, in fact, equivalent to $K_r$ being polynomially convex 
in $\{r\}\times\mathbb C^2$ for each $r\in B$. 


With the setup introduced in Section \ref{thirdpaper:sec:setup} we now set $ s_{r,j}:=\phi_r(\lambda_{r,j})$ and we set
$$
S_r:=\bigcup_{j=1}^n s_{r,j}\;.
$$
In the product space 
$B\times\mathbb C^2$ we define $S:=\{(r,(z,w)):(z,w)\in S_r\;,\;r\in B\}$.
\begin{prop}\label{thirdpaper:prop}
Let $K\subset (B\times\mathbb C^2)\setminus S$ be a compact set such that $K$ is 
polynomially convex. 
Let $T>0$, and let $\epsilon>0$. 
Then there exists a continuous map $g\colon B\times\mathbb C^2\rightarrow\mathbb C^2$
such that the following hold for all $r\in B$.
\begin{itemize}
\item[(i)] $g(r,\cdot)\in \Aut \mathbb C^2$,
\item[(ii)] $\|g(r,\cdot)-\id\|_{K_r}<\epsilon$, and
\item[(iii)] $g(r,S_r)\subset\mathbb C^2\setminus T\mathbb B^2$\;.
\end{itemize}
\end{prop}

For the following lemma we extend the map $\psi$ defined in Section \ref{thirdpaper:sec:setup} to a map $\psi\colon\mathbb P^1\times\mathbb C \to\mathbb P^1\times\mathbb C$ by setting $\psi(z,w):=(1/z + 1,w)$, and 
we extend the map $\phi_r(z)$ to a rational map on $\mathbb C^2$ by setting 
\begin{equation}\label{thirdpaper:emb1}
\phi_r(z,w) := \left(z,w +  \sum_{j=2}^n \frac{e^{i\theta_j}}{\alpha_{r,j}(z)}\right).
\end{equation}
Moreover for $T'<T''$ we set $S_r(T',T''):=\{(z,w)\in S_r: T'\leq |(z,w)|\leq T''\}$, and we let $S(T',T''):=\bigcup_r\{r\}\times S_r(T',T'')$ which is the union over $r$ in the product space $B\times\Bbb C^2$.
Finally define $S(T',T'')(\delta):=\bigcup_r\{r\}\times S_r(T',T'')(\delta)$ for $\delta>0$.

\begin{lem}\label{thirdpaper:isotopy}
There exist $T''>T'>>T$ arbitrarily large, $\delta>0$, such that for any $\epsilon>0$ there exists
an open set $U\subset B\times\Bbb C^2$ containing $S\cup \overline{S(T',T'')(\delta)}$ and 
a smooth fiber preserving map $\psi\colon[0,1]\times U\rightarrow B\times\mathbb C^2$
such that, for each $r\in B$ the following hold:
\begin{itemize}
\item[(i)] $\psi_{r,t}(\cdot)$ is an isotopy of holomorphic embeddings, and 
$\psi_{r,0}(\cdot)=\id$,
\item[(ii)] $\psi_{r,t}(S_r)\subset S_r$ for every $t\in[0,1]$,
\item[(iii)] $\|\psi_{r,t}-\id\|_{\mathscr C^2(S_r(T',T'')(\delta))}<\epsilon$\; for every $t\in[0,1]$, and
\item[(iv)] $\psi_{r,1}(S_r)\subset\mathbb C^2\setminus T\mathbb B^2$.
\end{itemize}
\end{lem}
\begin{proof}
Set $\gamma_{r,j}(z,w):=(b_j(r)\cdot z + a_j(r),w)$ such that $\gamma_{r,j}[-1,0]$ parametrizes $l_{r,j}$.
Setting $F_{r,j}:=\phi_r\circ\psi\circ \gamma_{r,j}$ we have that $F_{r,j}[-1,0]$ parametrizes $s_{r,j}$.
Fix $T>0$ and choose $-1<s<0$ such that 
$$
\bigcup_{r\in B}\bigcup_{j=1}^m F_{r,j}^{-1}((T+1)\overline{\mathbb B^2}\cap s_{r,j})\subset [-1,s].
$$
Choose any pair $T',T''$ such that $F_{r,j}^{-1}(s_{r,j}(T',T''))\subset (s,0)$ for all $r,j$.  \

\medskip

For $N\in\mathbb N$ define 
$$
\eta_{N,t}(z,w) := \left(\frac{z-t(1+s)e^{-N(z-s)} + t(1+s)e^{-N(-s)}}{1-t(1+s)e^{-N(-s)}},w\right). 
$$
Then $\eta_{N,t}$
is an isotopy of injective holomorphic maps near the real line in the $z$--plane, and
leaves the real line invariant, fixing $0$.  We see that 
$$
\eta_{N,1}(x,0) = \left(\frac{x-(1+s)e^{-N(x-s)}+ (1+s)e^{-N(-s)}}{1-(1+s)e^{-N(-s)}},0\right)
$$
from which $\eta_{N,1}(s,0)=(-1,0)$ and $\lim_{x\to+\infty}\eta_{N,1}(x,0)=(+\infty,0)$, so the interval $[s,\infty)$ is stretched to the interval $[-1,\infty)$ when $t=1$. Note that 
for any $s'>s$ we have that $\lim_{N\rightarrow\infty} \eta_{N,t} = \id$ uniformly 
on $\{\mathfrak{Re}(z)\geq s'\}$.  \

Now let $\sigma_{N,t}$ be the inverse isotopy to $\eta_{N,t}$, i.e., 
$\sigma_{N,t} = \eta_{N,1-t}\circ\eta_{N,1}^{-1}$; it is injective holomorphic near the 
real line in the $z$--plane, and by choosing $N$ large, may be extended, arbitrarily 
close to the identity, to any set $\{\mathfrak{Re}(z)\geq s'\}$ for $s'>s$.  \

We may now define $\psi_{r,t}(\cdot)$ on $s_{r,j}$ by  
$$
\psi_{r,t} :=  F_{r,j} \circ \sigma_{N,t} \circ F_{r,j}^{-1}\;.
$$
The claims of the lemma are satisfied by choosing $N$ large, and $\delta$ sufficiently small. 
\end{proof}

\begin{rk}\label{thirdpaper:polycvx}
If $\delta$ is sufficiently small and $\epsilon$ further sufficiently small we get that 
$$
K\cup \psi_t(S\cup \overline{S(T',T'')(\delta)})
$$
is polynomially convex. Observe first that $K_r\cup S_r(T',T'')$ is polynomially convex, since 
$K_r$ is, and $S_r(T',T'')$ is a collection of disjoint arcs. For a sufficiently small $\delta'$ it is known that 
the tube $\overline{S_r(T',T'')(\delta')}$ is polynomially convex, and for sufficiently 
small $\delta'$ we have that $K_r\cup \overline{S_r(T',T'')(\delta')}$ is polynomially convex. 
Then if $\delta<\delta'$ and we consider $\psi_{t,r}$ as in the lemma with $\delta'$
instead of $\delta$, if $\epsilon$ is small enough we get our claim, since the $\delta,\delta'$
may be chosen independently of $r$.
\end{rk}

\emph{Proof of Proposition \ref{thirdpaper:prop}:}
Fix $0<\delta<<1$.  For each $\eta\in \delta\mathbb B^2$
we set $v_\eta=(0,1) + \eta$, and we let $\pi_\eta$ denote the 
orthogonal projection onto the orthogonal complement of $v_\eta$.
After applying the linear transformation 
$$
A(z,w)=((1/\sqrt{2})z + (1/\sqrt{2})w,-(1/\sqrt{2})z + (1/\sqrt{2})w)
$$
it follows from Proposition \ref{thirdpaper:t-tang} that the family $\pi_{\eta}(s_{r,1})$
is $(\vartheta_{1,\eta}-\pi)$--directed and that the families $\pi_{\eta}(s_{r,j})$
are $(\vartheta_{2,\eta} + \theta_j -\pi)$--directed, where the $\vartheta_{j,\eta}$'s
vary continuously with $\eta$. From now on we will assume that have applied 
the transformation $A$ without changing the notation for all sets considered above.  

\medskip

By increasing $T>0$ we may assume that $K_r\subset T\mathbb B^2$ for all $r$, and we
fix  $R$ as in Theorem \ref{thirdpaper:Carleman} such that $\pi_\eta(T\mathbb B^2)\subset\triangle_R$, and choose 
$T'<T''$ such that $\pi_\eta(S_r(T',T''))\subset \mathbb C\setminus\triangle_{R+3+\frac32C}$ for all $r$ and all $\eta$. 

\medskip

Let $\psi_t$ be the isotopy from Lemma \ref{thirdpaper:isotopy}, extended to be the identity on some neighborhood of $K$ which 
we regard as being included in $U$.
On $\psi_{t_0}(U)$ we define the vector field
$X_{t_0}(\zeta)=\frac{d}{dt}_{t=t_0}\psi_t(\psi_{t_0}^{-1}(\zeta))$ (here $\zeta=(r,x)=(r,z,w)$).  The goal is to follow 
the standard Anders\'{e}n--Lempert procedure parametrically for approximating the flow of the time dependent vector field $X_t$
by compositions of flows of complete fields, but to modify these so that they do not move $S\setminus S(0,T'')$. 
The proof is the same as the corresponding proof in \cite{thirdpaper:KW18} where this was done without parameters, but 
we include here a sketch and some additional details. The reader is assumed to be familiar with the Anders\'{e}n--Lempert--Forstneri\v{c}--Rosay construction. 

\medskip

\emph{Step 1:}  We will find flows $\sigma_{r,j}(t,x), j=1,\dots,m$, such that the composition 
$$
\sigma_{r,m}\circ\cdots\circ\sigma_{r,1}(t,x)
$$
approximates $\psi_t$. The flows are of two forms:
\begin{equation}\label{thirdpaper:shear}
\sigma_{r,j}(t,x) = x +  t a_{r,j}(\pi_j(x))v_j
\end{equation}
or
\begin{equation}\label{thirdpaper:overshear}
\sigma_{r,j}(t,x)= x + (e^{ t a_{r,j}(\pi_j(x))} - 1)\langle x,v_j\rangle v_j. 
\end{equation}
We write $\sigma_{r,j}(t,x) = x +  b_{r,j}(t,x)v_j$.

\medskip

\emph{Step 2:} The plan is then roughly to find a family of cutoff functions $\chi_j\in \mathscr C^\infty(\mathbb C^2),\; 0\leq\chi_j\leq 1$, 
such that $\chi_j\equiv 1$ near $T'\overline{\mathbb B^2}$ and $\chi_j\equiv 0$ near $\mathbb C^2\setminus {T''\mathbb B^2}$, 
and define 
$$
\tilde\sigma_{r,j}(t,x) := x +  \chi_j(x)b_{r,j}(t,x)v_j, 
$$
in such a way that all compositions 
$$
\tilde\sigma(j)_r:=\tilde\sigma_{r,j}\circ\cdots\circ \tilde\sigma_{r,1}
$$
are as close to the identity as we like in $\mathscr C^1$--norm on $S_r(T',T'')(\delta/2)$.  Note that $\tilde\sigma_{r,j}=\sigma_{r,j}$
on $T'\overline{\mathbb B^2}$ and $\tilde\sigma_{r,j}=\id$ outside $T''\mathbb B^2$. 
In particular the families $\pi_j(\tilde\sigma(j)_r(S_r))$ are as close as we like 
to the original families $\pi_j(S_r)$ and identical outside some compact set. 

\medskip

\emph{Step 3:} For each $j$ we may rewrite $\tilde\sigma_{r,j}$ on $\tilde\sigma(j-1)_r(S_r)$ as
$$
\tilde\sigma_{r,j}(t,x) = x +  c_{r,j}(t,\pi_j(x))v_j \mbox{ or } \tilde\sigma_{r,j}(t,x) = x + (e^{ c_{r,j}(t,\pi_j(x))} - 1)\langle x,v_j\rangle v_j. 
$$
where the $c_{r,j}$'s extend to be holomorphic  near $\overline\triangle_{R+3+\frac32C}$ and zero on $\pi_j(\tilde\sigma(j-1)_r(S_r(T'',\infty)))$.

\medskip

\emph{Step 4:} Approximate the coefficients $c_{r,j}$ in the sense of Carleman using Theorem \ref{thirdpaper:Carleman}.

\medskip

We now include some estimates explaining why the above scheme works (see also \cite{thirdpaper:KW18} where 
the construction is done without dependence of parameters).

\medskip
Choose $\chi\in \mathscr C^{\infty}(\mathbb C^2)$ nonnegative such that $\chi\equiv 1$ near $T'\overline{\mathbb B^{2}}$
and $\chi\equiv 0$ near $\mathbb C^2\setminus T''\mathbb B^2$. \

We may assume that the vector fields $X_{r,t}$ satisfy $\|X_{r,t}\|<\alpha$ on $S(T',T'')(\delta)$ for any small $\alpha>0$. 
Thus, freezing the vector field at time $i/N$ to obtain a vector 
field $X^i_{}$ with a flow $\gamma^i_{r,t}$, we have that $\|\gamma^i_{r}(t/N,x)-x\|\leq (t/N)\alpha$
on $S_r(T',T'')(2\delta/3)$. Thus, we may assume that the compositions 
$$
\gamma^i_{r,t/N}\circ\cdots\circ\gamma^1_{r,t/N}
$$
exist and remain arbitrarily close to the identity on $S_r(T',T'')(\delta/2)$ for $i<N$. \

\medskip

By Remark \ref{thirdpaper:polycvx}  we may approximate each vector field $X^i_r$ to arbitrary precision by a polynomial vector field, which 
we will still denote by $X^i_r$. 
By the parametric Anders\'{e}n--Lempert observation, see Lemma 4.9.9 in \cite{thirdpaper:F17} and the proof of Theorem 2.3 in \cite{thirdpaper:K05}, the family of vector fields $X^i_{r}$ may be written as a sum of shear and over--shear vector 
fields $X^i_{r} = \sum_{j=1}^m Y^i_{r,j}$ with $Y^i_{r,j}(x)=g^i_{r,j}(x)v_j$ with $v_j=v_{\eta_j}$,  
with flows 
$$
\sigma^i_{r,j}(t,x) = x + b^i_{r,j}(t,x) v_j. 
$$
Write 
$$
\Theta^i_{r,t}(x) := \sigma^i_{r,m}(t,x)\circ\cdots  \circ \sigma^i_{r,1}(t,x) =: x + f^i_r(t,x).
$$
It is known that the composition 
\begin{equation}\label{thirdpaper:comp}
(\Theta^N_{r,t/nN})^n\circ \cdots \circ (\Theta^1_{r,t/nN})^n
\end{equation}
converges to the flow of $X_{r,t}$ on $K_r\cup S_r$ as $N$ and then $n$ tends to infinity.
Our first goal is to replace the maps $\sigma^i_{r,j}(t,x)$ by maps $x + \chi_j(x)b^i_{r,j}(t,x) v_j$
with $\chi_j(x)=1$ for $|x|\leq T'$ and $\chi_j(x)=0$ for $|x|\geq T''$ in the composition 
\eqref{thirdpaper:comp} or partial compositions of it, and show that we still get maps that are 
close to the identity on $S_r(T',T'')(\delta/2)$. The compositions thus obtained will remain 
the same on $\{|x|<T'\}$ and be the identity map outside $\{|x|\leq T''\}$.

\medskip

We will now modify the flows on $S_r(T',T'')(\delta)$.  We have that $\|f^i_r(t,x)\|\leq 2\alpha t$
for $t$ sufficiently small. If we set 
$$
\tilde\Theta^i_r(t,x) = x + \chi(x) f^i_r(t,x)
$$
we see that 
$$
\| (\tilde\Theta^i_{r,(t/n)})^m - \id\| \leq 2\alpha (m/n)t
$$
for $n$ large.   \

We decompose in a natural way 
$$
\sigma(j)_r(t,x):=\sigma_{r,j}\circ\cdots\circ\sigma_{r,1}(t,x) = x + h_{r,1}(t,x)v_1 + \cdots + h_{r,j}(t,x)v_j;
$$
then the $h_{r,j}(t,\cdot)$ go to zero as $t\rightarrow 0$. Now for large $n$ define
$$
\tilde\sigma_{r,1}(t/n,x)=x + \chi_1(x)h_{r,1}(t/n,x)v_1\;,
$$ 
and by induction 
$$
\tilde\sigma_{r,j+1}(t/n,x) = x + \chi_{j+1}(\tilde\sigma(j)_r^{-1}(t/n,x))h_{r,j+1}(\tilde\sigma(j)_r^{-1}(t/n,x))v_{j+1}.
$$

Then 
$$
\tilde\sigma^i_{r,m,t/n}\circ\cdots \circ \tilde\sigma^i_{r,1,t/n}(x) = \tilde\Theta^i_r(t/n,x),
$$
and we get that 
$$
\|(\tilde\Theta^N_{r,t/Nn})^n\circ\cdots\circ(\tilde\Theta^1_{r,t/Nn})^n-\id \|\leq 2\alpha t,
$$
and corresponding estimates hold for partial compositions.  Note that this shows that 
$S_r(T',T'')(\delta/2)$ remains in $S_r(T',T'')(2\delta/3)$ where we may assume that the 
$\mathscr C^1$--norms of the $f^i_r$'s are arbitrarily small, and so by arguments similar to those above
we get that all partial compositions are close to the identity in $\mathscr C^1$--norm, which will allow us to use the implicit 
function theorem to rewrite as in Step 3 above. \

\medskip

Finally, Step 4 is carried out exactly as in \cite{thirdpaper:KW18}.

$\hfill\blacksquare$

\section{Proof of Theorem \ref{thirdpaper:thm:main}}\label{thirdpaper:sec:main} 

\begin{proof}[Proof of Theorem \ref{thirdpaper:thm:main}]

Recall that $\psi\colon\mathbb P^1\rightarrow\mathbb P^1$ was originally defined by 
$\psi(z)=\frac{1}{z}+1$ and $\psi(\Bbb P^1\setminus L_r)=\Bbb P^1\setminus\Lambda_r$.
Then, if $0<\theta_2<\theta_3<\cdots<\theta_n$, 
$\phi_r\colon\Bbb C\setminus\{c_2(r),\dots,c_n(r)\}\rightarrow\mathbb C^2$ was defined by 
$$
\phi_r(z) = \left(z,  \sum_{j=2}^n \frac{e^{i\theta_j}}{\alpha_{r,j}(z)}\right).
$$
Let $S_r$ and $S$ be the sets defined in Section \ref{thirdpaper:sec:ALtheory} and set $X_r:=\phi_r\circ\psi(\mathbb P^1\setminus L_r)=\phi_r(\Bbb P^1)\setminus S_r$ which is a 1--dimensional complex manifold with boundary $\partial X_r=S_r$.
Define 
$$
C^r_j:=\mathbb P^1\setminus L_r(1/j), 
$$
such that $\{C_j^r\}_{j=1}^\infty$ is a normal exhaustion of $\mathbb P^1\setminus L_r$ by $\mathcal O(\Bbb P^1\setminus L_r)$--convex compact sets. It follows that $K^r_j:=\phi_r\circ\psi(C^r_j)$, $j\ge1$ is a normal exhaustion of $X_r$ by $\mathcal O(X_r)$--convex compact sets.
The proof of Proposition 1 in \cite{thirdpaper:W07} ensures the following two crucial facts:
\begin{enumerate}[label=(\roman*)]
\item\label{thirdpaper:pc} $K_j^{r}$ are polynomially convex, and\;
\item\label{thirdpaper:joint} given any $K\subset\C^2\setminus S_r$ compact polynomially convex, the set $K\cup K_j^{r}$ is polynomially convex for any $j$ large enough\;.
\end{enumerate}
We construct now inductively a sequence of continuous mappings, whose continuous limit
$$
h\colon \bigcup_{r\in B} \left(\{r\}\times X_r\right)\rightarrow\mathbb C^2
$$ 
will be a fiberwise proper holomorphic embedding that we will compose with suitable mappings to prove the statement.

Proposition \ref{thirdpaper:prop} provides a continuous $g_1\colon B\times\Bbb C^2\to \Bbb C^2$ such that, for every $r\in B$
\begin{itemize}
\item $g_1(r,\cdot)\in\operatorname{Aut}\C^2$,\;and
\item $g_1(r,S_r)\subset\C^2\setminus1\overline{\B^2}$.
\end{itemize}

\medskip

Assume that we have constructed  $H_j\colon B\times\Bbb C^2\to B\times \Bbb C^2$, $H_j(r,\cdot)=(r,h_j(r,\cdot))$, continuous such that for every $r\in B$ we have that
\begin{itemize}
\item $h_j(r,\cdot)\in\operatorname{Aut}\C^2$,\;and
\item $h_j(r,S_r)\subset\C^2\setminus j\overline{\B^2}$\;.
\end{itemize}

It follows from \ref{thirdpaper:joint} that 
$$
L_j^{r}:=h_j(r,K_{m_j}^{r})\cup j\overline{\B^2}\subset\C^2\setminus h_j(r,S_r)\;.
$$
is polynomially convex for sufficiently large $m_j$. 
Then for every $\epsilon_j>0$, Proposition \ref{thirdpaper:prop} gives us $g_{j+1}\colon B\times\Bbb C^2\to \C^2$ continuous, such that for every $r\in B$ the following hold:
\begin{itemize}
\item $g_{j+1}(r,\cdot)\in\operatorname{Aut}\C^2$\;,
\item $\|g_{j+1}(r,\cdot)-\id\|_{L_j^{r}}<\epsilon_j$,\;and
\item  $g_{j+1}(r,h_j(r,S_r))\subset\C^2\setminus(j+1)\overline{\B^2}$\;.
\end{itemize}
Define $G_{j+1}(r,\cdot):=(r,g_{j+1}(r,\cdot))$ and consequently $H_{j+1}:=G_{j+1}\circ H_j$. 
Then letting $\epsilon_j\to 0$ and $m_j\to+\infty$ fast enough, the push--out method (see \cite{thirdpaper:F17}) allows to conclude that, for every $r\in B$, the sequence $\{h_j(r,\cdot)\}_j$ converges uniformly on compact subsets of $D_r:=\bigcup_{j\ge1}h_j^{-1}(r,L_j^r)$ to a biholomorphism $h_r\colon D_r\to\Bbb C^2$. It is straightforward to check that $X_r\subset D_r\subset\Bbb C^2\setminus S_r$, from which it follows that $S_r=\partial  X_r\subseteq\partial D_r$, thus $h_r\colon X_r\to\Bbb C^2$ is a proper holomorphic embedding for every $r\in B$. So setting $H(r,\cdot):=h_r(\cdot),\;\Psi(r,\cdot):=(r,\psi(\cdot))$, and $\Phi(r,\cdot):=(r,\phi_r(\cdot))$, the mapping $\Xi:=H\circ\Phi\circ\Psi\colon\Omega\to\Bbb C^2$ proves the claim of the Theorem. 

\end{proof}

\newpage
\bibliographystyle{amsplain}

\end{document}